\documentclass[12pt,a4paper,leqno]{amsart}

\title[Schr\"odinger type propagators and modulation spaces]{Schr\"odinger type propagators, pseudodifferential operators and modulation spaces}

\author[E. Cordero]{Elena Cordero}

\address{Dipartimento di Matematica,  Universit\`a  di Torino, Via Carlo Alberto 10, 10123 Torino, Italy}

\email{elena.cordero@unito.it}

\author[A. Tabacco]{Anita Tabacco}

\address{Dipartimento di Scienze Matematiche, Politecnico di Torino, Corso Duca degli Abruzzi 24, 10129 Torino, Italy}

\email{anita.tabacco@polito.it}

\author[P. Wahlberg]{Patrik Wahlberg}

\address{Dipartimento di Matematica,  Universit\`a  di Torino, Via Carlo Alberto 10, 10123 Torino, Italy}

\email{patrik.wahlberg@unito.it}



\date{\today}

\usepackage{amsmath,amssymb,amsfonts,amsthm,amsopn,dsfont}
\usepackage{graphics}
\usepackage{amscd,amsxtra}
\usepackage{latexsym}
\usepackage{enumerate}

\newtheorem{theorem}{Theorem}[section]
\newtheorem{corollary}[theorem]{Corollary}

\newtheorem{definition}[theorem]{Definition}

\newtheorem{lemma}[theorem]{Lemma}
\numberwithin{equation}{section}

\newtheorem{proposition}[theorem]{Proposition}
\newtheorem{remark}[theorem]{Remark}

\renewcommand{\S}{\mathcal{S}}

\newcommand{\beqa}{\begin{eqnarray*}}
\newcommand{\eeqa}{\end{eqnarray*}}

\newcommand{\field}[1]{\mathbb{#1}}
\newcommand{\bR}{\field{R}}        
\newcommand{\bN}{\field{N}}        
\def\o{\eta}

 \def\cF{\mathcal{F}}              
 \def\cS{\mathcal{S}}

 \def\cM{\mathcal{M}}

 \def\cW{\mathcal{W}}

\def\a{\aleph}

\def\rd{\bR^d}

\def\rdd{{\bR^{2d}}}

\def\lrd{L^2(\rd)}

\def\intrd{\int_{\rd}}

\def\R{\right)}
\def\<{\left<}
\def\>{\right>}

\def\mv1{M_v^1}

\def\phas{(x,\o )}
\def\mn{(m,n)}
\def\mn'{(m',n')}

\def\o{\eta}
\def\a{\alpha}

\def\R{\mathbb{R}}
\def\Ren{\mathbb{R}^d}
\def\Renn{\mathbb{R}^{2d}}

\def\sch{\mathcal{S}}

\def\Fur{\mathcal{F}}

\def\f{\varphi}

\def\Sn2{S_{2}(L^{2}(\Ren))}
\def\S1{S_{1}(L^{2}(\Ren))}
\def\sig00{\sigma_{0,0}}

\def\la{\langle}
\def\ra{\rangle}

\begin{document}


\begin{abstract}
We prove continuity results for Fourier integral operators with
symbols in modulation spaces, acting between modulation spaces.
The phase functions belong to a class of nondegenerate generalized quadratic forms
that includes Schr\"odinger propagators and pseudodifferential operators.
As a byproduct we obtain a characterization of all exponents $p,q,r_1,r_2,t_1,t_2 \in [1,\infty]$ of modulation spaces such that a symbol in $M^{p,q}(\rdd)$ gives a pseudodifferential operator that is continuous from $M^{r_1,r_2}(\rd)$ into $M^{t_1,t_2}(\rd)$.
\end{abstract}

\maketitle

\section{Introduction}\label{intro}
Fourier integral operators (FIOs) represent a mathematical tool to study the behavior of the solutions to  partial
differential equations. Our type of FIOs has its origins in Quantum Mechanics: they arise naturally in the study of the
Cauchy problem for Schr\"odinger-type operators. We refer the reader to the pioneering works of Asada and Fujiwara~\cite{asada-fuji}, Cordoba and Fefferman~\cite{CF78}, and Helffer and Robert~\cite{helffer-rob1}.
This paper is concerned with the study of FIOs formally defined by
\begin{equation}\label{fio}
    Tf(x)=\intrd e^{2\pi i \Phi\phas}
    \sigma\phas
    \hat{f}(\o)d\o.
\end{equation}
The functions $\sigma$ and $\Phi$ are called symbol (or amplitude) and phase function, respectively.
Our phase functions $\Phi$, sometimes called ``tame'' \cite{CGN,CGFR12},  are real-valued, smooth functions on $\rdd$,
satisfying $\,\partial_z^\a \Phi \in L^\infty (\rdd )$ for $\a \geq 2$, and  the non-degeneracy condition
\begin{equation}\label{detcond}
   \left|{\rm det}\,
\left(\frac{\partial^2\Phi}{\partial
x_i\partial \eta_l}\Big|_{
(x,\eta)}\right)\right|\geq
\delta>0\quad \forall\,
(x,\eta)\in \R^{2d}.
\end{equation}
Basic examples are provided by quadratic forms in the variables $x,\eta\in\rd$ and the corresponding FIOs are the so called generalized metaplectic operators \cite{CGFR12,folland}.
Another well-known example is the phase $\Phi\phas=x\cdot \o$ which gives  pseudodifferential operators in the Kohn-Nirenberg form.
Note that these phase functions differ from those of FIOs arising in the solutions
of hyperbolic equations, that are positively homogeneous of degree one in
$\eta$ (see e.g.\cite{cordero-nicola-rodino,hormander71,ruzhsugimoto,seegersoggestein}).

The aim of this paper is to provide  optimal boundedness results for FIOs of the type above having rough symbols. The symbol classes that are suitable for this study reveal to be the so-called modulation spaces, introduced by Feichtinger in 1983 \cite{F1} and recalled in  Subsection \ref{2.1} below.
Modulation spaces will be employed both for  symbol spaces and  spaces on which operators act.

Sharpness results in this framework were already pursued in the papers \cite{fiomodulation1,fio1,Sjostrand1}, where symbols in the particular modulation space $M^{\infty,1}(\rdd)$ were considered. Other results in this connection are contained in \cite{concetti1,concetti2, tataru,toftconcetti1}.

The special case of pseudodifferential operators has been studied in the context of modulation spaces
by several authors, including the earlier works by
Gr\"ochenig and Heil \cite{GH99,GH04}, Labate \cite{Labate1,Labate2},
Sj\"ostrand \cite{Sjostrand1}, Tachizawa \cite{Tachizawa1}. Recent  contributions are provided by \cite{benyi,benjiKasso,corderonicola4,Czaja,sugitomita2,Toft,Toftweight}.
 For simplicity, let us first present our results in terms of pseudodifferential operators. The following sufficient conditions
 enlarge Toft's conditions \cite[Theorem 4.3]{Toft}, whereas the necessary conditions contain those in \cite[Proposition 5.3]{corderonicola4}.

\begin{theorem}\label{Charpseudo}
Assume that $1\leq p,q,r_i,t_i\leq \infty$, $i=1,2$.
Then the
pseudodifferential operator $T$, from $\cS(\rd)$ to $\cS'(\rd)$,
having symbol  $\sigma \in M^{p,q}(\R^{2d})$, extends uniquely to
a bounded operator from $\mathcal{M}^{r_1,r_2}(\R^d)$ to
$\mathcal{M}^{t_1,t_2}(\R^d)$, with the estimate
\begin{equation}\label{stimaA}
\|Tf\|_{\mathcal{M}^{t_1,t_2}} \lesssim
\|\sigma\|_{M^{p,q}}\|f\|_{\mathcal{M}^{r_1,r_2}}
\end{equation}
if and only if
\begin{equation}\label{indicitutti}
1/r_i - 1/t_i \geq 1-1/p-1/q, \quad i=1,2,
\end{equation}
\noindent
and
\begin{equation}\label{indiceq}
\quad q \leq \min(t_1,t_2,r_1',r_2').
\end{equation}
\end{theorem}
This result can be seen as a characterization of pseudodifferential operators acting on modulation spaces, which completes the previous studies on this topic.

The sufficient conditions are obtained as a corollary of more general results for FIOs, contained in Theorem \ref{boundedphasegradient} below. Let us give an overview of our results in this framework.

Our main theme is to derive interpolation-theoretic consequences of the boundedness results for FIOs in \cite{fiomodulation1,fio1} and their possible sharpness.
These prior results treat symbols in the modulation space $M^{\infty,1}(\rdd)$, possibly with a spatial weight or additional constraints on the phase function, and the continuity of the corresponding FIOs acting either  on $\mathcal{M}^{r_1,r_2}(\rd)$ or from $\mathcal{M}^{r_1,r_2}(\rd)$ to $\mathcal{M}^{r_2,r_1}(\rd)$. We consider more general modulation spaces $M^{p,q}(\rdd)$ as symbol classes and studying the action from $\mathcal{M}^{r_1,r_2}(\rd)$ to $\mathcal{M}^{t_1,t_2}(\rd)$, $1\leq r_i,t_i\leq\infty$, $i=1,2$.

First we show that a symbol in $M^{\infty}(\rdd)$ gives rise to a FIO that maps $\mathcal{M}^{1}(\rd)$ into $ \mathcal{M}^{\infty}(\rd)$ continuously, and a symbol in $M^{1}(\rdd)$ gives rise to a FIO that maps $\mathcal{M}^{\infty}(\rd)$ into $\mathcal{M}^{1}(\rd)$ continuously. These results are similar to results by Concetti, Garello and Toft \cite{concetti1,concetti2,toftconcetti1}.

Using complex interpolation and the results of \cite{fiomodulation1} we then deduce continuity of FIOs with symbols in $M^{p,q}(\rdd)$ acting from $\mathcal{M}^{r_1,r_2}(\rd)$ to $\mathcal{M}^{t_1,t_2}(\rd)$, and search for the weakest possible conditions on the family of exponents $p,q,r_1,r_2,t_1,t_2 \in [1,\infty]$ that admit continuity.
If we make the additional assumption on the phase function
\begin{equation}\label{phasegrad}
\sup_{x,x',\eta\in\rd}\left| \nabla_x\Phi(x,\eta)-\nabla_x\Phi(x',\eta)\right|<\infty,
\end{equation}
then the corresponding  FIO  $T$ is continuous and satisfies \eqref{stimaA} if and only if \eqref{indicitutti} and \eqref{indiceq} hold, see Theorem \ref{boundedphasegradient} and Remark \ref{3.9}. Note that \eqref{phasegrad}
is satisfied in the special case of $\Phi\phas=x\cdot\eta$, i.e. $T$ is a pseudodifferential operator.

If we omit the assumption \eqref{phasegrad}  and study the action on  spaces $\mathcal{M}^{r_1,r_2}$, with $r_1\not=r_2$,  then the behavior of a FIO $T$ is more troublesome.
For instance, let us study the boundedness of $T$ on $\mathcal{M}^{r_1,r_2}$, with $r_1\not=r_2$.
Consider the pointwise multiplication operator  $Tf(x)=e^{\pi i |x|^2}f(x)$,
which can be seen as a FIO with phase function
$\Phi(x,\eta)=x \cdot \eta+|x|^2/2$ (that does not satisfy \eqref{phasegrad}),
and symbol $\sigma\equiv1\in
M^{\infty,1}(\rdd)$. Taking $t_i=r_i$, $i=1,2$, the conditions \eqref{indicitutti} and \eqref{indiceq} are satisfied
 but the operator $T$  is  bounded on $\mathcal{M}^{r_1,r_2}$ if and only if $r_1=r_2$, cf. \cite[Proposition
7.1]{fio1}.

Nevertheless, if we do not assume \eqref{phasegrad} we can still obtain  continuity on
$\mathcal{M}^{r_1,r_2}$ for all $r_1,r_2 \in [1,\infty]$, provided we introduce weights on the symbol spaces such that the symbols  decay faster  at infinity. Our main result in this direction is provided by Theorem \ref{main} below.

Finally, motivated by the search for fixed-time estimates for one-parameter Schr\"odinger-type propagators (see \cite[Section 4]{cordero2} and \cite[Section 5]{fiomodulation1}), we study in detail the action of a Fourier integral operator $T$  from the spaces  $\mathcal{M}^{r_1,r_2}$ into $\mathcal{M}^{r_2,r_1}$, $r_2\leq r_1$ (and analogously for Wiener amalgam spaces). We end  by discussing the sharpness of the results. This topic is detailed in Section \ref{3.2}.

\subsection*{Notation}

The Schwartz space is denoted by $\cS(\rd)$ and the tempered distributions by $\cS'(\rd)$.
The Fourier transform of $f \in \cS(\rd)$ is normalized as
$
\cF f(\eta) = \hat f(\eta) = \int_{\rd} f(x) e^{-2\pi i x \cdot \eta} dx.
$
For $s\in\R$ and $x \in \R^d$ we
set $v_s(x) = \la x \ra^s = (1+|x|^2)^{s/2}$, and $x \cdot \eta$ denotes the inner product on $\R^d$.
The notation $f \lesssim g$ means that there exists a positive constant $C>0$ such that
$f \leq C g$ (uniformly over all arguments of $f$ and $g$ where appropriate), while $f \asymp g$ means $f \lesssim g$ and $g \lesssim f$.
Translations are denoted by $T_x f(y) = f(y-x)$ and
modulations by $M_\o f(x) = e^{2 \pi i x \cdot \o} f(x)$, $x,y,\o \in \rd$, $f \in \cS(\rd)$.
The inner product on $L^2(\rd)$ is conjugate linear in the second argument and is denoted by $\la \cdot,\cdot \ra$,
which also denotes  the conjugate linear action of $\cS'$ on $\cS$.

\section{Preliminaries}\label{preliminaries}

In order to emphasize that $T$ defined by \eqref{fio} depends on $\sigma$ we sometimes write $T=T_\sigma$.
\begin{definition}\label{deffase}
A real-valued phase functions $\Phi$ is called  \emph{tame} (\cite{CGN,fiomodulation1,fio1}) provided the following three conditions are satisfied:
\begin{enumerate}[(i)]
\item $\Phi\in C^{\infty}(\rdd)$;
\item there exist constants $C_\alpha>0$ such that
\begin{equation}\label{phasedecay}
|\partial^\a \Phi(x,\eta)|\leq C_\a \quad  \forall \a \in \bN^{2d}, \quad |\a|\geq
2;
\end{equation}
\item $\Phi$ satisfies the non-degeneracy condition \eqref{detcond}.
\end{enumerate}
\end{definition}

\subsection{Modulation
spaces}\cite{F1,fe89-1,feichtinger90,
Feichtinger-grochenig89a,Feichtinger-grochenig89b,grochenig,triebel83}\label{2.1}

In order to define  modulation spaces we use the short-time Fourier transform (STFT) $V_g f$
of a tempered distribution $f\in\cS'(\rd)$ with respect
to a nonzero window $g\in\cS(\rd)$. It is defined as
$V_g f(x,\o) = \mathcal F (f \, T_x \overline{g}) (\o)$ for $x,\eta \in \rd$.
For $f \in \cS(\rd)$ we have
\begin{equation}\label{STFT}
V_g f(x,\o)=\la f,M_\o T_xg\ra= \intrd e^{-2\pi i \o \cdot
y}f(y) \, \overline{g(y-x)}\,dy, \quad x,\eta \in \R^d.
\end{equation}
The inversion formula for the STFT (see e.g. (\cite[Corollary 3.2.3]{grochenig}) reads:
If
$\|g\|_{L^2}=1$ and $f \in \cS(\R^d)$ then
\begin{equation}\label{treduetre}
f=\int_{\R^{2d}} V_g f(x,\o) \, M_\o T_x g \, dx \, d\o.
\end{equation}

The following property of the STFT \cite[Lemma
11.3.3]{grochenig} is useful when one needs to change window
function.
\begin{lemma}\label{changewind}
If $f\in\cS'(\rd)$, $g_0,g_1,\gamma\in\cS(\rd)$ and $\la \gamma,
g_1\ra\not=0$, then
$$|V_{g_0} f\phas|\leq\frac1{|\la\gamma,g_1\ra|}(|V_{g_1} f|\ast|V_{g_0}\gamma|)\phas, \quad x,\o \in \rd.
$$
\end{lemma}

In order to define the weighted modulation spaces of the symbols, we first introduce the class $\mathcal{M}_{v_{s}}(\rdd)$, $s\geq0$, consisting of weights $m$  that are  positive measurable functions on $\rdd$ and satisfy $m(x+y) \lesssim v_s(x) m(y)$, $x,y \in \rdd$.
It follows that $v_t$ is $v_s$-moderate for all $t \in \mathbb R$ such that $|t| \leq s$.
In particular, we shall consider
the class of weight functions on $\rdd$ given by $v_{s_1,s_2}\phas=\la x\ra^{s_1}\la
\o\ra^{s_2}$, $s_1, s_2 \in\R$, $x, \eta \in \rd$, and
$m=v_{s_1,s_2}\otimes 1$ on $\R^{4d}$.

Given a window function $g\in\sch(\Ren) \setminus \{ 0 \}$,
$m \in \mathcal{M}_{v_{s}}(\rdd)$ for some $s \geq 0$,
and $1\leq p,q\leq \infty$, the {\it modulation space} $M^{p,q}_m(\rd)$ consists of all tempered
distributions $f\in\sch'(\Ren)$ such that $V_gf\in L^{p,q}_m(\rdd)$ (weighted mixed-norm Lebesgue space). The norm on $M^{p,q}_m$ is
$$
\|f\|_{M^{p,q}_m}=\|V_gf\|_{L^{p,q}_m}=\left(\int_{\Ren}
  \left(\int_{\Ren}|V_gf(x,\o)|^pm(x,\o)^p\,
    dx\right)^{q/p}d\o\right)^{1/q}  \,
$$
(with natural modifications when $p=\infty$ or $q=\infty$). If $p=q$, we
write $M^p_m$ instead of $M^{p,p}_m$, and if $m\equiv 1$ on
$\Renn$, then we write $M^{p,q}$ and $M^p$ for $M^{p,q}_m$ and
$M^{p,p}_m$, respectively. The space  $M^{p,q}_m (\rd)$ is a Banach space
whose definition is independent of the choice of the window $g$,
in the sense that different  nonzero window functions yield equivalent  norms.
The modulation space $M^{\infty,1}$ is also called Sj\"ostrand's class \cite{Sjostrand1}.

The closure of $\cS(\rd)$ in the  $M^{p,q}_m$-norm is denoted $\mathcal{M}^{p,q}_m(\rd)$.
Then $\mathcal{M}^{p,q}_m(\rd) \subseteq M^{p,q}_m(\rd)$, and $\mathcal{M}^{p,q}_m (\rd) = M^{p,q}_m (\rd)$
provided $p<\infty$ and $q<\infty$.
For any $p,q \in [1,\infty]$ and any $m \in \mathcal{M}_{v_{s}}(\rdd)$, $s \geq 0$, we have:
The inner product $\la \cdot,\cdot \ra$ on $\cS (\rd) \times \cS (\rd)$ extends to a continuous sesquilinear map
$M^{p,q}_m (\rd) \times M^{p',q'}_{1/m} (\rd) \rightarrow \mathbb C$.
Here and elsewhere the conjugate exponent $p'$ of $p \in [1,\infty]$ is defined by $1/p+1/p'=1$.
If
\begin{equation}\nonumber
\| f \| = \sup | \la f,g \ra |,
\end{equation}
with supremum taken over all $g \in \cS(\rd)$ such that $\| g \|_{M^{p',q'}_{1/m}} \leq 1$,
then $\| \cdot \|$ and $\| \cdot \|_{M^{p,q}_m}$ are equivalent norms (cf. \cite[Proposition 1.4 (3)]{Toftweight}).
This result will be invoked using the phrase ``by duality''.

Suppose $m_1, m_2 \in \mathcal{M}_{v_{s}}(\rdd)$ for some $s \geq 0$. Then we have the embeddings
\begin{equation}\label{modspaceincl1}
\begin{aligned}
& \cS (\rd) \subseteq M_{m_1}^{p_1,q_1} (\rd) \subseteq M_{m_2}^{p_2,q_2} (\rd) \subseteq
\sch'(\rd), \\
& p_1 \leq p_2, \quad q_1 \leq q_2, \quad m_2 \lesssim
m_1.
\end{aligned}
\end{equation}

Modulation spaces are closed under complex interpolation as follows (cf. \cite{fei83b}).
If $p_j,q_j \in [1,\infty]$, $m_j \in \mathcal{M}_{v_{s}}(\rdd)$, $j=1,2$, for some $s \geq 0$,
$0 < \theta < 1$,
\begin{align*}
\frac1{p} = \frac{1-\theta}{p_1} + \frac{\theta}{p_2}, \quad
\frac1{q} = \frac{1-\theta}{q_1} + \frac{\theta}{q_2}, \quad
m=m_1^{1-\theta} m_2^\theta,
\end{align*}
then
\begin{equation}\label{interpolation}
\left( \cM_{m_1}^{p_1,q_1} (\rd), \cM_{m_2}^{p_2,q_2} (\rd) \right)_{[\theta]} = \cM_{m}^{p,q} (\rd).
\end{equation}

We need the following result concerning the modulation space norm of distributions
of compact support in time or in frequency (cf., e.g.,
\cite{concetti1,fe89-1,kasso07,A11}).
\begin{lemma}\label{lloc} Let $1\leq p,q\leq
\infty.$\\
(i) For every $u\in \mathcal{S}'(\rd)$, supported in a compact set
$K \subseteq \rd$, we have $u\in M^{p,q}\Leftrightarrow u\in \Fur
L^q$, and
\begin{equation}\label{loc}
C_K^{-1} \|u\|_{M^{p,q}}\leq \|u\|_{\cF L^q}\leq C_K
\|u\|_{M^{p,q}},
\end{equation}
where $C_K>0$ depends only on
 $K$.\\
 (ii) For every $u\in \mathcal{S}'(\rd)$,
whose Fourier transform is supported in a compact set $K \subseteq
\rd$, we have $u\in M^{p,q}\Leftrightarrow u\in L^p$, and
\begin{equation}\label{loc2}
C_K^{-1} \|u\|_{M^{p,q}}\leq\|u\|_{ L^p}\leq C_K\|u\|_{M^{p,q}},
\end{equation}
where $C_K>0$ depends only on
 $K$.\\
\end{lemma}
We refer to Gr\"ochenig's book \cite{grochenig} for further properties of the modulation spaces.

Parseval's formula gives $|V_g f(x,\o)|=|V_{\hat g} \hat f(\o,-x)| = |\mathcal F (\hat f \, T_\o \overline{\hat g}) (-x)|$ for $f \in \cS'(\rd)$ and $g \in \cS(\rd)$.
Hence
$$
\| f \|_{{M}^{p,q}} = \left( \int_{\rd} \| \hat f \ T_{\o} \overline{\hat g} \|_{\cF L^p}^q \ d \o \right)^{1/q}
= \| \hat f \|_{W(\cF L^p,L^q)}.
$$
Here $W(\cF L^p,L^q)(\rd)$ are  particular cases of  \emph{Wiener amalgam spaces} with local component $\cF L^p(\rd)$ and global component $L^q(\rd)$. It follows that we have $\cF ({M}^{p,q})=W(\cF L^p,L^q)$.
The closure of $\cS(\rd)$ in the  $W(\cF L^p,L^q)$-norm is denoted $\cW(\cF L^p,L^q)$.
For more information on Wiener amalgam spaces we refer to \cite{fei83a,fei83b,fournier-stewart85,Heil03}.

We will need the modulation space norm of a complex Gaussian.

\begin{lemma}\label{lemm2}
For $a>0$, $b\in\R$,  set
$h_{a+ib}(x)=e^{-\pi(a+ib)|x|^2}$.
Then we have for all $1\leq p,q\leq\infty$
\begin{equation}\label{lem2}
\|h_{a+ib}\|_{M^{p,q}}  \asymp
\frac{\left((a+1)^2+b^2\right)^{
\frac{d}{2}\left(\frac{1}{p}-\frac{1}{2}
\right)}}{a^\frac{d}{2q}\left(a(a+1)+b^2\right)^{\frac{d}{2}\left(\frac{1}{p}-\frac{1}{q}\right)}}.
\end{equation}
\end{lemma}
\begin{proof}
For $G_{a+ib}(x)=(a+ib)^{-d/2}e^{-\frac{\pi|x|^2}{a+ib}}$ we have by \cite[Lemma 2.9]{fiomodulation1}
\begin{equation}\nonumber
\| G_{a+ib} \|_{W(\Fur
L^p,L^q)}\asymp
\frac{\left((a+1)^2+b^2\right)^{
\frac{d}{2}\left(\frac{1}{p}-\frac{1}{2}
\right)}}{a^\frac{d}{2q}\left(a(a+1)+b^2\right)^{\frac{d}{2}\left(\frac{1}{p}-\frac{1}{q}\right)}},
\end{equation}
for all $1\leq p,q \leq\infty$.
Thus we obtain from $\cF ({M}^{p,q})=W(\cF
L^p,L^q)$
$$
\|h_{a+ib}\|_{M^{p,q}}
= \|\widehat{h_{a+ib}}\|_{W(\cF L^p,L^q)}
= \| G_{a+ib} \|_{W(\cF L^p,L^q)}.
$$
\end{proof}

In particular we recover Toft's result \cite[Lemma 1.8]{Toft}.
If $\varphi(x) = e^{- \pi |x|^2}$ and $\varphi_\lambda (x) = \varphi(\lambda x)$ then
\begin{equation}\label{gaussdilmod}
\| \varphi_\lambda \|_{M^{p,q}} \asymp \lambda^{-d/p} (1+\lambda)^{d(1/p+1/q-1)}, \quad \lambda>0.
\end{equation}

\subsection{Continuity of Fourier integral operators on modulation spaces}

Here we recollect and add comments on the results on Fourier integral operators and modulation spaces upon which the results in this paper build.

Assume that the phase function $\Phi$ is tame and satisfies the  condition \eqref{phasegrad}.
Then a symbol that belongs to  Sj\"ostrand's class $M^{\infty,1}(\rdd)$
gives rise to an operator that is continuous on $\mathcal{M}^{r_1,r_2}(\rd)$, for every $1\leq
r_1,r_2\leq\infty$.
More precisely, the following \cite[Theorem 1.1]{fiomodulation1} holds.

\begin{theorem}\label{cont} Consider a
tame phase function $\Phi$ satisfying \eqref{phasegrad}, and a symbol $\sigma\in M^{\infty,1}(\rdd)$.
Then the corresponding FIO \ $T$ extends to a
bounded operator on $\mathcal{M}^{r_1,r_2}(\rd)$, for every $1\leq
r_1,r_2\leq\infty$, with the estimate
\begin{equation}\nonumber
\|Tf\|_{\mathcal{M}^{r_1,r_2}}\lesssim\|\sigma\|_{M^{\infty,1}}\|f\|_{\mathcal{M}^{r_1,r_2}}.
\end{equation}
\end{theorem}
Note that the pseudodifferential operator phase function $\Phi(x,\eta) = x \cdot \eta$ satisfies the assumptions
of Theorem \ref{cont}.
If we omit the assumption \eqref{phasegrad} we can still get continuity on
$\mathcal{M}^{r_1,r_2}$ for all $r_1,r_2 \in [1,\infty]$, provided we introduce weights on the symbols according to the following result \cite[Theorem 1.2]{fiomodulation1}.

\begin{theorem}\label{elefabio}
Consider a tame phase function $\Phi$,  a symbol $\sigma\in
M_{v_{s_1,s_2}\otimes1}^{\infty,1}(\rdd)$,
$s_1,s_2\in\R$ and $1\leq r_1,r_2\leq\infty$.
Assume one of the following conditions:

\begin{enumerate}[(i)]
\item $r_1=r_2$ and $s_1,s_2\geq0$,
\item $r_2<r_1$, $s_1>d\left(\frac{1}{r_2}-\frac{1}{r_1}\right)$
and $s_2\geq0$,
\item $r_1<r_2$, $s_1\geq0$ and
$s_2>d\left(\frac{1}{r_1}-\frac{1}{r_2}\right)$.
\end{enumerate}
Then the corresponding FIO \ $T$ extends to a bounded operator on $\mathcal{M}^{r_1,r_2}(\rd)$, and
\begin{equation*}
\|Tf\|_{\mathcal{M}^{r_1,r_2}}\lesssim\|\sigma\|_{M_{v_{s_1,s_2}\otimes1}^{\infty,1}}\|f\|_{\mathcal{M}^{r_1,r_2}}.
\end{equation*}
\end{theorem}

The next quoted result \cite[Theorem 1.3]{fiomodulation1} shows, in particular, that a further condition on the phase function \eqref{d2fix}
gives a FIO that is continuous from $\mathcal{M}^{r_1,r_2}(\rd)$ into $\mathcal{M}^{r_2,r_1}(\rd)$.

\begin{theorem}\label{elefabioqppq}
Consider a tame phase function $\Phi$, and let $1\leq r_2 \leq r_1\leq\infty$.
Assume one of the following conditions:
\begin{enumerate}[(i)]
\item $s_1,s_2 \geq 0$, the symbol $\sigma\in
M^{\infty,1}(\rdd)$ and for some $\delta>0$,
\begin{equation}\label{d2fix}
   \left|{\rm det}\,
\left(\frac{\partial^2\Phi}{\partial
x_i\partial x_l}\Big|_{
(x,\eta)}\right)\right|\geq
\delta\quad \forall
(x,\eta)\in \R^{2d},
\end{equation}
\item the symbol
$\sigma\in M^{\infty,1}_{v_{s_1,s_2}\otimes1}(\rdd)$, with
$s_1 > d\left(\frac{1}{r_2}-\frac{1}{r_1}\right)$ and $s_2 \geq 0$.
\end{enumerate}
Then the corresponding FIO \ $T$ extends to a bounded operator
from $\mathcal{M}^{r_1,r_2}(\rd)$ into $\mathcal{M}^{r_2,r_1}(\rd)$,
and
\begin{equation}\nonumber
\|Tf\|_{\mathcal{M}^{r_2,r_1}}\lesssim\|\sigma\|_{M_{v_{s_1,s_2}\otimes1}^{\infty,1}}\|f\|_{\mathcal{M}^{r_1,r_2}}.
\end{equation}
\end{theorem}

A typical tame phase function that satisfies \eqref{d2fix}
is $\Phi(x,\eta)=x \cdot \eta + |x|^2/2$. When the symbol is $\sigma \equiv 1 \in M^{\infty,1}$,
the FIO is the pointwise multiplication operator
\begin{equation}\label{pointmult}
T f(x) = e^{i \pi |x|^2} f(x), \quad x \in \rd.
\end{equation}
By Theorem \ref{elefabioqppq} and \cite[Proposition 6.6]{fiomodulation1},
continuity from $\mathcal{M}^{r_1,r_2}(\rd)$ into $\mathcal{M}^{r_2,r_1}(\rd)$ holds for this operator
if and only if $r_2 \leq r_1$.
Continuity of the operator \eqref{pointmult} from $\mathcal{M}^{r_1,r_2}(\rd)$ into $\mathcal{M}^{t_1,t_2}(\rd)$
is equivalent to continuity of the Schr\"odinger multiplier operator
\begin{equation}\label{schrodmult}
Tf(x)=\intrd e^{2\pi i x \cdot \eta} e^{\pi i |\eta |^2} \hat{f}(\eta)
\,d\eta, \quad f\in\cS(\rd),
\end{equation}
from $\cW(\cF L^{r_1},L^{r_2})$ into $\cW(\cF L^{t_1},L^{t_2})$.

For $1\leq r_1<r_2\leq\infty$, one could conjecture
that a FIO \ $T_\sigma$ is bounded from
$\mathcal{M}^{r_1,r_2}$ to $\mathcal{M}^{r_2,r_1}$,
provided $\sigma\in M^{\infty,1}(\rdd)$, and the phase function  $\Phi$ is tame and satisfies
\begin{equation}\label{detcond01}
   \left|{\rm det}\,
\left(\frac{\partial^2\Phi}{\partial \eta_i\partial \eta_l}\Big|_{
(x,\eta)}\right)\right|\geq \delta\quad \forall (x,\eta)\in
\R^{2d},
\end{equation}
for some $\delta>0$, instead of \eqref{d2fix}.
But the conjecture is false as shown by the following result.
It treats the Schr\"odinger multiplier \eqref{schrodmult}, which is a FIO with phase
$\Phi\phas=x \cdot \eta + |\eta|^2/2$ and symbol $\sigma\equiv 1\in
M^{\infty,1}$.

\begin{proposition} The  Schr\"odinger multiplier \eqref{schrodmult} is bounded from $\mathcal{M}^{r_1,r_2}(\rd)$ to $\mathcal{M}^{t_1,t_2}(\rd)$ if and only if $r_i\leq t_i$, $i=1,2$.
\end{proposition}
\begin{proof}
The sufficiency of the condition follows immediately by combining the boundedness result on $\mathcal{M}^{r_1,r_2}(\rd)$ provided by \cite[Theorem 1]{benyi}, and the inclusion relations for modulation spaces \eqref{modspaceincl1}.
Vice versa, assume that
$$\|Tf\|_{\mathcal{M}^{t_1,t_2}(\rd)} \lesssim \|f\|_{\mathcal{M}^{r_1,r_2}(\rd)}, \quad f\in\cS(\rd).
$$
Taking $f=\f_\lambda(t)=e^{-\pi \lambda^2 |t|^2}$, \eqref{gaussdilmod} gives
\begin{equation}\label{dilgauss}
\|\f_\lambda\|_{\mathcal{M}^{r_1,r_2}(\rd)}\asymp\left\{\begin{array}{ll}
\lambda^{-\frac{d}{r_1}}, & \lambda\to 0\\
\lambda^{-\frac{d}{r_2'}}, & \lambda\to +\infty.
\end{array}
\right.
\end{equation}
Straightforward computations give
$$
T\f_\lambda = (1-i \lambda^2)^{-d/2} e^{-\pi \frac{\lambda^2}{1-i \lambda^2} |\cdot|^2},
$$
and an application of Lemma \ref{lemm2} yields
\begin{equation}\nonumber
\|T\f_\lambda\|_{\cM^{t_1,t_2}}
\asymp
\left\{\begin{array}{ll}
\lambda^{-\frac{d}{t_1}}, & \lambda\to 0\\
\lambda^{-\frac{d}{t_2'}}, & \lambda\to +\infty.
\end{array}
\right.
\end{equation}
Combining with \eqref{dilgauss} we obtain $r_1\leq t_1$ for $\lambda \to 0$,
 and $r_2\leq t_2$ for $\lambda\to+\infty$.
\end{proof}

\subsection{Fourier integral operators and Wiener amalgam spaces}

From Theorems \ref{cont}, \ref{elefabio} and \ref{elefabioqppq}
we may infer results on continuity of FIOs acting on Wiener amalgam spaces.
Indeed, since $\cF \mathcal{M}^{r_1,r_2}=\mathcal{W}(\cF L^{r_1},
L^{r_2})$, it follows  that
\begin{equation}\label{duality1}
\| T f \|_{\mathcal{W}(\cF L^{t_1},
L^{t_2})}
\lesssim
\|\sigma\|_{M_{v_{s_1,s_2} \otimes 1}^{p,q}}\|f\|_{\mathcal{W}(\cF L^{r_1},
L^{r_2})}
\end{equation}
if and only if
\begin{equation*}
\| \tilde T f \|_{\mathcal{M}^{t_1,t_2}}
\lesssim
\|\sigma\|_{M_{v_{s_1,s_2} \otimes 1}^{p,q}}\|f\|_{\mathcal{M}^{r_1,r_2}},
\end{equation*}
where $\tilde{T} = \Fur \circ T \circ \Fur^{-1}$. By duality and an explicit computation this is equivalent to veryfing that the adjoint  operator
\begin{equation}\label{Ttildeexpr}
\tilde{T}^\ast f(x)=\int e^{-2\pi i\Phi(-\eta,x)}\overline{\sigma(-\eta,x)}\hat{f}(\eta)\, d\eta
\end{equation}
extends to a bounded operator  from $\mathcal{M}^{t'_1,t'_2}(\rd)$ to $\mathcal{M}^{r'_1,r'_2}(\rd)$.
Since  $\sigma \in M_{v_{s_1,s_2} \otimes 1}^{p,q}(\rdd)$ if and only if  $\tilde \sigma(x,\eta)=\overline{\sigma(-\eta,x)}\in M_{v_{s_2,s_1} \otimes 1}^{p,q}(\rdd) $
(cf. \cite[Lemma 2.10]{fiomodulation1}),
the continuity statement for Wiener amalgam spaces in \eqref{duality1}  follows from  the  continuity estimate
\begin{equation}
\| \tilde T^\ast_{\tilde\sigma} f \|_{\mathcal{M}^{r'_1,r'_2}}
\lesssim
\|\tilde \sigma\|_{M_{v_{s_2,s_1} \otimes 1}^{p,q}}\|f\|_{\mathcal{M}^{t'_1,t'_2}}.
\end{equation}

These considerations immediately transfer continuity results for FIOs acting on modulation spaces
to FIOs acting on Wiener amalgam spaces, possibly with modified assumptions on the phase function, cf.
\cite[Corollary 3.9 and Corollary 5.2]{fiomodulation1}.
We shall state continuity results for FIOs acting on Wiener amalgam spaces as corollaries
of the corresponding continuity results for modulation spaces.

\subsection{A characterization of modulation spaces}

We  recall the  formula, obtained in \cite[Section
6]{fio1} (see also \cite[Proposition 3.2]{fiomodulation1} and \cite[Section 4]{CGFR12}), which
expresses the Gabor matrix of the FIO $\ T$ in terms of the STFT of
its symbol $\sigma$.
Suppose the phase function $\Phi$ satisfy {\it (i)} and {\it (ii)} of Definition \ref{deffase}.
Choose a non-zero window function
$g\in \cS(\R^d)$, and define for $z,\zeta\in\rdd$
\begin{equation}\label{zpsi}
\Psi_{z}(\zeta):=e^{2\pi
i\Phi_{2,z}(\zeta)}(\bar{g}\otimes\hat{g})(\zeta),
\end{equation}
where
\begin{equation}\label{taylorrest}
\Phi_{2,z}(\zeta)=2\sum_{|\a|=2}\int_0^1(1-t)\partial^\a
\Phi(z+t\zeta)\,dt\frac{\zeta^\a}{\a!}, \quad z,\zeta\in\rdd.
\end{equation}
Let $\sigma \in \cS'(\rdd)$ and denote $g_{x,\o} = M_\o T_x g$ for $x,\o\in\R^d$. Then the so called Gabor matrix of $T$, given by $\la T g_{x,\o}, g_{x',\o'}\ra$, can be expressed via the STFT as
\begin{equation}\label{symbstft}
|\la T g_{x,\o}, g_{x',\o'}\ra|
=   \ |V_{\Psi_{(x',\o)}} \sigma(x',\o,\o'-\nabla_x\Phi(x',\o),x-\nabla_\eta \Phi(x',\o))|,
\end{equation}
for $x, \o, x', \o' \in \rd.$

We present a characterization of the spaces $M^{p,q}_{1\otimes m}(\R^{2d})$ when $m\in \mathcal{M}_{v_{s}}(\rdd)$ for $s \geq 0$.
This is a generalization of \cite[Proposition 3.10]{CGN}, that treats the cases $(p,q)=(\infty,1)$ and $p=q=\infty$,
to general $p,q \in [1,\infty]$. This characterization shows that the the time-frequency concentration of the symbol $\sigma$ does not depend on the parameter $z\in\rdd$ of the window $\Psi _z$, defined in \eqref{zpsi}.

First, we need the following simplified version of \cite[Lemma 3.9]{CGN}.

\begin{lemma}\label{wienerprop}
If $s \geq 0$, $g\in\cS(\rd)$ and $\Psi\in\cS(\rdd)$ then
\begin{equation}\label{l1W}
\sup_{z\in\rdd}|V_{\Psi_{z}}\Psi| \in L^1_{1\otimes
v_{s}}(\R^{4d}).
\end{equation}
\end{lemma}
The characterization for modulation spaces in terms of $\Psi _z$ is as follows.
\begin{proposition}\label{proposition4}
Let $s\geq0$, $m\in \mathcal{M}_{v_{s}}(\rdd)$, $p,q \in [1,\infty]$, and
$\sigma\in\cS'(\rdd)$.  Then
$$\sigma \in M^{p,q}_{1\otimes
 m}(\R^{2d}) \quad \Longleftrightarrow \quad \sup_{z\in\rdd}|V_{\Psi_z}
 \sigma|\in L^{p,q}_{1\otimes m}(\R^{4d}),$$
and
\begin{align}\label{p1}
\|\sigma\|_{M^{p,q}_{1\otimes
 m}(\R^{2d})}&\asymp \|\sup_{z\in\rdd}|V_{\Psi_z}
\sigma|\,\|_{L^{p,q}_{1\otimes
    m}(\R^{4d})} \\
    &= \left(\int_{\rdd} \left(\int_{\rdd}\sup_{z\in\rdd}|V_{\Psi_z}
\sigma(u_1,u_2)|^p m(u_2)^p d u_1\right)^{\frac q
p}\,du_2\right)^{\frac1q}\nonumber
\end{align}
(with obvious modifications when $p=\infty$ or $q=\infty$).
\end{proposition}
\begin{proof}  We first prove
$\|\sup_{z\in\rdd}|V_{\Psi_z} \sigma|\,\|_{ L^{p,q}_{1\otimes m}}\lesssim \|\sigma\|_{M^{p,q}_{1\otimes m}}$.
Taking $\Psi\in\cS(\rdd)$ such that $\| \Psi \|_{L^2} = 1$ and using Lemma \ref{changewind}, we have
$$
|V_{\Psi_z}\sigma|(u_1,u_2)|\leq |V_{\Psi}\sigma| \ast|V_{\Psi_z} \Psi|(u_1,u_2)\leq |V_{\Psi}\sigma| \ast\sup_{z\in\rdd} |V_{\Psi_z} \Psi|(u_1,u_2).
$$
Young's inequality and the assumption $m\in\mathcal{M}_{v_{s}}(\rdd)$ yield
$$
\|\sup_{z\in\rdd} |V_{\Psi_z}
\sigma|\,\|_{L^{p,q}_{1\otimes m}}\leq \|V_{\Psi}\sigma\|_{L^{p,q}_{1\otimes m}}\|\sup_{z\in\rdd} |V_{\Psi_z}
\Psi|\,\|_{L^1_{1\otimes v_s}}
\lesssim \|\sigma\|_{M^{p,q}_{1\otimes m}},
$$
thanks to Lemma \ref{wienerprop}.

On the other hand, assume
$\|\sup_{z\in\rdd}|V_{\Psi_z} \sigma|\,\|_{L^{p,q}_{1\otimes m}(\R^{4d})}<\infty$.
Since $\| \Psi_z \|_{L^2} = \|g\|_{L^2}^4$ and $|V_{\Psi}\Psi_z(u)| = |V_{\Psi_z}\Psi (-u)|$,  denoting $\widetilde f(x)=f(-x)$, Lemma \ref{changewind} gives
\begin{align*}
|V_{\Psi}\sigma (u_1,u_2)|
& \lesssim |V_{\Psi_z}\sigma|\ast| V_{\Psi}\Psi_z |(u_1,u_2) \\
& = |V_{\Psi_z}\sigma|\ast| \widetilde{V_{\Psi_z}\Psi} |(u_1,u_2) \\
& \leq \left(\sup_{z\in\rdd} |V_{\Psi_z}\sigma |\right)\ast \left( \sup_{z\in\rdd} |\widetilde{V_{\Psi_z}\Psi}| \right)(u_1,u_2).
\end{align*}
Applying Young's inequality and Lemma \ref{wienerprop}, we finally obtain
\begin{align*}
\|\sigma\|_{M^{p,q}_{1\otimes m}}
& = \|V_{\Psi}\sigma\|_{L^{p,q}_{1\otimes m}}
\lesssim \|\sup_{z\in\rdd}|V_{\Psi_z}\sigma|\|_{L^{p,q}_{1\otimes m}}\|
\sup_{z\in\rdd} | \widetilde{V_{\Psi_z}\Psi} |\|_{L^{1}_{1\otimes v_{s}}} \\
& \lesssim \|\sup_{z\in\rdd}| V_{\Psi_z}\sigma |\|_{L^{p,q}_{1\otimes m}}.
\end{align*}
\end{proof}

\section{Continuity results for FIOs}\label{FIOcont}

First we will prove two results concerning continuity of FIOs with symbols in $M^\infty(\R^{2d})$ and $M^1(\R^{2d})$, respectively.
Then  we will make complex interpolation between them and the results in Section \ref{preliminaries}.

We need the following Schur-type test, whose proof is obvious.

\begin{lemma}\label{proschur}
Consider an integral operator $A$ on $\R^{2d}$, given by
\[
(A f)(x',\o')=\iint_{\rdd} K(x',\o';x,\o)f(x,\o)\,dx\,d\o.
\]
(i) If $K\in L^\infty(\R^{4d})$ then $A$ is continuous
from $L^1(\rdd)$ into $L^{\infty}(\rdd)$.\\
(ii) If $K\in L^1(\R^{4d})$ then $A$ is continuous
from $L^\infty(\rdd)$ into $L^{1}(\rdd)$.
\end{lemma}

\begin{proposition}\label{Minfty}
Consider a tame phase function $\Phi$  and suppose $\sigma \in M^\infty(\R^{2d})$. Then
$T_\sigma$ extends to a bounded operator from $\mathcal{M}^1(\R^d)$ to
$\mathcal{M}^\infty(\R^d)$, and
\begin{equation}\label{caseMinftyinfty}
\| T_\sigma f \|_{\mathcal{M}^\infty(\R^d)} \lesssim \| \sigma
\|_{M^{\infty}(\R^{2d})} \| f \|_{\mathcal{M}^1(\R^d)}.
\end{equation}
\end{proposition}
\begin{proof}
Let $g \in \cS(\rd)$ with $\| g \|_{L^2}=1$.
For $\varphi \in \cS(\rd)$, the map $f \rightarrow \la T f, \varphi \ra$, denoted $u_\varphi$, belongs to $\cS'(\rd)$.
Since $u_\varphi$ is linear (rather than antilinear) we obtain from the inversion formula \eqref{treduetre} and \cite[Theorem 5.1.1]{hormander1}
\begin{align*}
\la T f, \varphi \ra
& = \la u_\varphi, \overline f \ra
= \la u_\varphi, \int_{\rdd} \overline{V_g f (x,\eta) M_\eta T_x g (\cdot)} \, dx \, d\eta \ra \\
& = \int_{\rdd} \la u_\varphi, \overline{M_\eta T_x g} \ra \, V_g f (x,\eta)  \, dx \, d\eta \\
& = \int_{\rdd} \la T \, M_\eta T_x g, \varphi \ra \, V_g f (x,\eta) \, dx \, d\eta.
\end{align*}
It follows that, for $f \in \cS(\rd)$,
\[
V_g (Tf)(x',\o')=\int_{\R^{2d}}\langle Tg_{x,\o},
g_{x',\o'}\rangle \, V_g f(x,\o) \, dx\,d\o.
\]
The desired estimate \eqref{caseMinftyinfty} thus follows if we can prove that the map $K_T$
defined by
\[
K_T G(x',\o')=\int_{\R^{2d}}\langle Tg_{x,\o},
g_{x',\o'} \rangle \, G(x,\o) \, dx\,d\o
\]
is continuous from $L^1(\rdd)$ into $L^\infty(\rdd)$. By Lemma
\ref{proschur} \emph{(i)} it suffices to prove that its integral kernel
\[
K_T(x',\o';x,\o)=\langle Tg_{x,\o}, g_{x',\o'}\rangle
\]
satisfies $ K_T\in L^\infty(\R^{4d})$. By \eqref{symbstft} we have
\begin{align*}
|K_T(x',\o';x,\o)|&=|V_{\Psi_{(x',\o)}}\sigma(x',\o,\o'-\nabla_x\Phi(x',\o),x-\nabla_\eta
\Phi(x',\o))|\\&\leq \sup_{z\in\rdd}
|V_{\Psi_{z}}\sigma(x',\o,\o'-\nabla_x\Phi(x',\o),x-\nabla_\eta,
\Phi(x',\o))|
\end{align*}
and hence
\begin{align*}
& \sup_{(x,\o,x',\o')\in\R^{4d}}|K_T(x',\o';x,\o)| \\
& \leq\sup_{(x,\o,x',\o')\in\R^{4d}}
\sup_{z\in\rdd}
|V_{\Psi_{z}}\sigma(x',\o,\o'-\nabla_x\Phi(x',\o),x-\nabla_\eta
\Phi(x',\o))|\\
&=\sup_{(x,\o,x',\o')\in\R^{4d}} \sup_{z\in\rdd}
|V_{\Psi_{z}}\sigma(x',\o,\o',x)|\asymp \|\sigma\|_{M^{\infty}}
\end{align*}
by the characterization \eqref{p1}.
\end{proof}

Proceeding similarly as in the proof of Proposition \ref{Minfty}, using
Lemma \ref{proschur} $(ii)$ instead of $(i)$, gives the following dual result.

\begin{proposition}\label{M1}
Consider a tame phase function $\Phi$  and suppose $\sigma \in M^1(\R^{2d})$. Then $T_\sigma$
extends to a bounded operator from $\mathcal{M}^\infty(\R^d)$ to
$\mathcal{M}^1(\R^d)$, and
\begin{equation}\label{caseM11}
\| T_\sigma f \|_{\mathcal{M}^1(\R^d)} \lesssim \| \sigma
\|_{M^{1}(\R^{2d})} \| f \|_{\mathcal{M}^\infty(\R^d)}.
\end{equation}
\end{proposition}

\begin{remark}
We notice that Proposition \ref{Minfty} and Proposition \ref{M1} hold with weaker assumptions on the real-valued phase function.
In fact, conditions (i), (ii) and (iii) of Definition \ref{deffase} may evidently be relaxed to $\Phi \in C^\infty(\R^{2d})$
and $\sup_{|\alpha|=2} |\partial^\alpha \Phi| \lesssim v_N$ for some $N>0$.
A similar result, with assumptions on $\Phi$ that are weaker than (i), (ii) and (iii) but stronger than $\Phi \in C^2(\R^{2d})$ and $\sup_{|\alpha|=2} |\partial^\alpha \Phi| \lesssim v_{N}$,
is shown in \cite[Theorem 2.7]{concetti2}.
More precisely, \cite[Theorem 2.7]{concetti2} treats a more general type of FIO whose phase function depends on three variables as
\begin{equation}\nonumber
Tf(x) = \iint_{\R^{2 d}} e^{2\pi i \, \varphi(x,y,\xi)} \ \sigma(x,\xi) \ f(y) \ dy \ d\xi.
\end{equation}
Specializing to our situation, we have $\varphi(x,y,\xi) = \Phi(x,\xi)-y \cdot \xi$, and the sufficient condition on $\Phi$ in \cite[Theorem 2.7]{concetti2} is $\partial ^\alpha \Phi \in M^{\infty,1}$ for all $\alpha \in \mathbb N^d$ such that $|\alpha|=2$.

There is also a version of this result for weighted modulation spaces in \cite[Proposition 3.1 (3)]{toftconcetti1}.
The symbol space, as well as the spaces between which the operator acts, are then weighted modulation spaces, with polynomially bounded weights that are related as described in \cite[Proposition 3.1 (3)]{toftconcetti1}.
\end{remark}

\subsection{Results based on Theorems \ref{cont} and \ref{elefabio}}

Propositions \ref{Minfty} and \ref{M1} admit us to prove the following interpolation-theoretic consequences of
Theorems \ref{cont} and \ref{elefabio}.
First we discuss the case when both the domain and the range are equal-index modulation spaces.

\begin{theorem}\label{fiointerpolation1}
Consider a tame phase function $\Phi$,
and let $1 \leq p,q,r,t \leq \infty$.
If
\begin{equation}\label{indexcondition1}
q \leq \min(t,r') \quad \mbox{and} \quad 1/r-1/t \geq 1-1/p-1/q,
\end{equation}
and $\sigma \in M^{p,q}(\R^{2d})$, then $T$ extends to a bounded operator from
$\mathcal M^r(\R^d)$ to $\mathcal M^t(\R^d)$, with
\begin{equation}\label{contstatement1}
\| T f \|_{\mathcal M^t(\rd)} \lesssim \| \sigma \|_{M^{p,q}(\rdd)} \| f \|_{\mathcal M^r(\rd)}.
\end{equation}
\end{theorem}
\begin{proof}
If $\sigma \in M^{\infty,1}(\R^{2d})$ then $T$ extends, according to Theorem \ref{elefabio} $(i)$, to a bounded operator on $\mathcal M^s(\R^d)$ for all $1 \leq s \leq \infty$ and
\begin{equation}\label{caseMinfty0}
\| T_\sigma f \|_{\mathcal M^s(\R^d)} \lesssim \| \sigma \|_{M^{\infty,1}(\R^{2d})} \| f \|_{\mathcal M^s(\R^d)}.
\end{equation}
Regarding $T$ as the bilinear map $(\sigma,f) \mapsto T f$, \eqref{caseMinfty0} and \eqref{caseMinftyinfty} of Proposition \ref{Minfty} says that $T$ is continuous
\begin{equation*}
\begin{aligned}
M^{\infty,1}(\R^{2d}) & \times \mathcal M^s(\R^d) & \to & \quad
\mathcal M ^s(\R^d) \quad \mbox{for} \quad 1 \leq s \leq \infty, \quad \mbox{and} \\
M^\infty(\R^{2d}) & \times \cM^1(\R^d) & \to & \quad \cM^\infty(\R^d)
\end{aligned}
\end{equation*}
Using multi-linear complex interpolation (cf. \cite[Theorem 4.4.1]{bergh-lofstrom})
and \eqref{interpolation}, it follows that the bilinear map $T$ is continuous
\begin{equation}\label{caseMinftyq}
T : \cM^{\infty,q}(\R^{2d}) \times \mathcal M^r(\R^d) \to \mathcal M^t(\R^d),
\end{equation}
for $q,r,t \in [1,\infty]$ such that $1/r - 1/t = 1-1/q$, $q\leq \min(t,r')$ and $r \leq t$.
Likewise, interpolation between \eqref{caseMinftyq} and \eqref{caseM11} of Proposition \ref{M1} gives
\eqref{contstatement1} for $p,q,r,t \in [1,\infty]$ such that
$$
q \leq \min(p,t,r') \quad \mbox{and} \quad 1/r-1/t = 1-1/p-1/q.
$$
Due to the embeddings \eqref{modspaceincl1}, we may relax these assumptions on $r$ and $t$ (possibly decreasing $r$ and increasing $t$), keeping $p,q$ fixed, into
\begin{equation}\label{assumption1}
q \leq \min(p,t,r') \quad \mbox{and} \quad 1/r-1/t \geq 1-1/p-1/q,
\end{equation}
and \eqref{contstatement1} still holds true.
Finally, again using the embeddings \eqref{modspaceincl1} in order to relax the conditions on $p$ and $q$, possibly decreasing $p$ and $q$ while keeping $r,t$ fixed, it can be verified that the result extends to all $p,q,r,t \in [1,\infty]$ such that
$$
q \leq \min(t,r'), \quad 1/r-1/t \geq 1-1/p-1/q.
$$
\end{proof}

\begin{remark}
For $p=\infty$ and $1 < q < \infty$, the sufficient condition on the symbol in Theorem \ref{fiointerpolation1} should be
$\sigma \in \cM^{\infty,q}(\rdd)$ rather than $\sigma \in M^{\infty,q}(\rdd)$.
This small modification is understood in all results of this paper, but not spelled out in order not to burden the presentation.
\end{remark}

\begin{remark}
The sufficient conditions in Theorem \ref{boundedphasegradient} are also necessary. Indeed, choose $\Phi(x,\eta) = x \cdot \eta$ which is tame. Then the corresponding  FIO $T$ reduces to a pseudodifferential operator and the necessary conditions are provided by Theorem \ref{Charpseudo}, written for the case $r_1=r_2=r$, $t_1=t_2=t$.
\end{remark}

\begin{remark}
We notice that a related result for weighted modulation spaces follows from a combination of \cite[Proposition 1.10]{toftconcetti1} and \cite[Theorem 2.10]{toftconcetti1}.
In particular, it follows  that a version of the inequality \eqref{caseMinfty0} holds for weighted spaces and symbols,
for certain combinations of weights, and $1 < s < \infty$,
when the phase function $\Phi$ is tame.
However, we inform the reader that the condition on the phase function
\begin{equation}\nonumber
   \left|{\rm det}\,
\left(\frac{\partial^2 \varphi}{\partial
y_i\partial \xi_l}\Big|_{
(x,y,\xi)}\right)\right| \geq \delta \quad \forall
(x,y,\xi)\in \R^{3d}
\end{equation}
for some $\delta>0$,
specified as sufficient for the conclusions in \cite[Proposition 3.1]{toftconcetti1}, is correct
for parts (1), (2) and (3) of that proposition, but not for part (4).
\end{remark}

Next we treat FIOs acting on modulation spaces from $\cM^{r_1,r_2}$ to $\cM^{t_1,t_2}$ with possibly $r_1 \neq r_2$ or $t_1 \neq t_2$. In this case the assumptions of Theorem \ref{fiointerpolation1} are not enough to provide boundedness (see the counterexample in the Introduction). Instead we obtain  results by strenghtening either the phase (Theorem \ref{boundedphasegradient}) or the symbol (Theorem \ref{main}) hypotheses.

Since the arguments of the proof of the result below follow closely the proof of Theorem \ref{fiointerpolation1}, starting from Theorem \ref{cont} and using Propositions \ref{Minfty} and \ref{M1}, we omit its proof.

\begin{theorem}\label{boundedphasegradient}
Let $1 \leq p,q,r_1,r_2,t_1,t_2 \leq \infty$, let the phase
function $\Phi$ be tame and satisfy \eqref{phasegrad},
and suppose \eqref{indicitutti} and \eqref{indiceq} hold true.
If $\sigma\in M^{p,q}(\rdd)$ then the corresponding operator $T$ extends to a bounded
operator from $\mathcal M^{r_1,r_2}(\R^d)$ to $\mathcal
M^{t_1,t_2}(\R^d)$, with
\begin{equation}\label{normaoperatore7}
\|Tf\|_{\mathcal{M}^{t_1,t_2}} \lesssim
\|\sigma\|_{M^{p,q}}\|f\|_{\mathcal{M}^{r_1,r_2}}.
\end{equation}
\end{theorem}
\begin{remark}\label{3.9}
The sufficient conditions in Theorem \ref{boundedphasegradient} are also necessary. Indeed, choose $\Phi(x,\eta) = x \cdot \eta$ which is tame and satisfies \eqref{phasegrad}. Then the corresponding  FIO $T$ reduces to a pseudodifferential operator and the necessary conditions are provided by Theorem \ref{Charpseudo}.
\end{remark}

As a consequence of Theorem  \ref{boundedphasegradient} we obtain the following result for Wiener amalgam spaces.
\begin{corollary}
Under the assumptions of Theorem \ref{boundedphasegradient}, with \eqref{phasegrad} replaced by
\begin{equation}\label{fase-bis}
\sup_{x,\eta,\eta'\in\rd}\left|
\nabla_{\eta}\Phi(x,\eta)-\nabla_{\eta}\Phi(x,\eta')\right|<\infty,
\end{equation}
the corresponding operator $T$ extends to a bounded operator from the space  $\mathcal{W}(\cF L^{r_1},L^{r_2})(\rd)$ to $\mathcal{W}(\cF L^{t_1},L^{t_2})(\rd)$, with
\begin{equation}\label{normaoperatore3Wpq}
\| T f \|_{\mathcal{W}(\cF L^{t_1},L^{t_2})}
\lesssim\| \sigma \|_{M^{p,q}} \| f \|_{\mathcal{W}(\cF L^{r_1},L^{r_2})}.
\end{equation}
\end{corollary}

If we do not assume the condition \eqref{phasegrad} on the phase, then similar FIO boundedness results can still be obtained by asking for more decay at infinity of the corresponding symbol. This means that we replace the unweighted modulation spaces $M^{p,q}$ by  weighted spaces.

\begin{theorem}\label{main}
Let $1 \leq p,q,r_1,r_2,t_1,t_2 \leq \infty$ and suppose \eqref{indicitutti} holds.
Consider a tame phase function $\Phi$, and a
symbol $\sigma\in M_{v_{s_1,s_2}\otimes1}^{p,q}(\rdd)$,
$s_1,s_2\in\R$. Suppose furthermore that either of the following two requirements are satisfied.

\begin{equation*}
\begin{aligned}
(i) \quad & s_2 \geq 0, \quad \mbox{and \ either} \\
& q \leq \min (t_1,t_2,r_1'), \quad r_2 \leq r_1, \quad \mbox{and} \quad s_1>d(1/r_2-1/r_1), \\
& \mbox{or} \\
& q \leq \min (t_2,r_1',r_2'), \quad t_2 \leq t_1, \quad \mbox{and} \quad s_1>d(1/t_2-1/t_1).
\end{aligned}
\end{equation*}
\begin{equation*}
\begin{aligned}
(ii) \quad & s_1 \geq 0, \quad \mbox{and \ either} \\
& q \leq \min (t_1,t_2,r_2'), \quad r_1 \leq r_2, \quad \mbox{and} \quad s_2 > d(1/r_1-1/r_2), \\
& \mbox{or} \\
& q \leq \min (t_1,r_1',r_2'), \quad t_1 \leq t_2, \quad \mbox{and} \quad s_2 > d(1/t_1-1/t_2).
\end{aligned}
\end{equation*}
Then $T$ extends to a bounded operator from $\mathcal
M^{r_1,r_2}(\R^d)$ to $\mathcal M^{t_1,t_2}(\R^d)$, and
\begin{equation*}
\|Tf\|_{\mathcal{M}^{t_1,t_2}} \lesssim
\|\sigma\|_{M_{v_{s_1,s_2} \otimes 1}^{p,q}}\|f\|_{\mathcal{M}^{r_1,r_2}}.
\end{equation*}
\end{theorem}

\begin{proof} The boundedness follows by complex interpolation, using  Theorem \ref{elefabio}, Propositions \ref{Minfty} and \ref{M1}, as detailed below.
By Theorem \ref{elefabio} $(i)$ and $(ii)$ we have
\begin{equation}\nonumber
\|Tf\|_{\mathcal{M}^{r_1,r_2}} \lesssim \|\sigma\|_{M_{v_{s_1,0} \otimes 1}^{\infty,1}}\|f\|_{\mathcal{M}^{r_1,r_2}}
\end{equation}
for $1 \leq r_2 \leq r_1 \leq \infty$ and $s_1 > d(1/r_2-1/r_1)$.
Proposition \ref{Minfty} and interpolation give continuity of
\begin{equation}\label{caseMinftyq1}
T : \cM_{v_{s,0} \otimes 1}^{\infty,q}(\R^{2d}) \times \mathcal M^{r_1,r_2}(\R^d) \to \mathcal M^{t_1,t_2}(\R^d),
\end{equation}
for $q,r_1,r_2,t_1,t_2 \in [1,\infty]$ and $s \in \R$ such that
$1/r_i - 1/t_i = 1-1/q$ for $i=1,2$, $r_2 \leq r_1$, $q \leq \min(t_2,r_1')$ and $s > d(1/r_2-1/r_1)$.
Next interpolation between \eqref{caseMinftyq1} and Proposition \ref{M1} gives
\begin{equation}\label{intermediateresult12}
\| T f \|_{\mathcal{M}^{t_1,t_2}}\lesssim\|\sigma\|_{\cM_{v_{s,0}\otimes1}^{p,q}}\|f\|_{\mathcal{M}^{r_1,r_2}}
\end{equation}
for $p,q,r_1,r_2,t_1,t_2 \in [1,\infty]$ and $s \in \R$ that satisfy
\begin{equation*}
\begin{aligned}
& q\leq \min (p,t_2,r_1'), \quad 1/r_i-1/t_i =
1-1/p-1/q \quad \mbox{for} \quad i=1,2, \\
& r_2 \leq r_1, \quad \mbox{and} \quad s>d(1/r_2-1/r_1).
\end{aligned}
\end{equation*}
Invoking \eqref{modspaceincl1}, we may relax the conditions on $t_i$, $r_i$ for $i=1,2$.
Thus we may possibly increase $t_i$ and decrease $r_i$ for $i=1,2$ while keeping $p,q$ fixed, such that
\eqref{indicitutti} holds, and either
\begin{equation*}
q \leq \min (p,t_1,t_2,r_1'), \quad r_2 \leq r_1, \quad \mbox{and} \quad s>d(1/r_2-1/r_1),
\end{equation*}
or
\begin{equation*}
q \leq \min (p,t_2,r_1',r_2'), \quad t_2 \leq t_1, \quad \mbox{and} \quad s>d(1/t_2-1/t_1),
\end{equation*}
while preserving \eqref{intermediateresult12}.
Again using \eqref{modspaceincl1}, these conditions may be further relaxed, in the sense of possibly decreasing $p$ and $q$ while keeping $r_i,t_i$, $i=1,2$, fixed, such that \eqref{indicitutti} holds, and either
\begin{equation*}
q \leq \min (t_1,t_2,r_1'), \quad r_2 \leq r_1, \quad \mbox{and} \quad s>d(1/r_2-1/r_1),
\end{equation*}
or
\begin{equation*}
q \leq \min (t_2,r_1',r_2'), \quad t_2 \leq t_1, \quad \mbox{and} \quad s>d(1/t_2-1/t_1),
\end{equation*}
while maintaining \eqref{intermediateresult12}.
Finally, another appeal to \eqref{modspaceincl1} shows that
$M_{v_{s_1,s_2}\otimes1}^{p,q}(\rdd) \subseteq M_{v_{s_1,0}\otimes1}^{p,q}(\rdd)$ which
proves the Theorem under assumption $(i)$.

If we instead use the assumption $(ii)$, the theorem is proved with a similar argument,
replacing Theorem \ref{elefabio} $(ii)$ by Theorem \ref{elefabio} $(iii)$ at the beginning.
\end{proof}

\begin{corollary}
Consider a phase $\Phi$ and a symbol $\sigma$ satisfying the assumptions of Theorem \ref{main}.
Then the corresponding operator $T$ extends to a bounded operator from $\mathcal{W}(\cF L^{r_1},L^{r_2})(\rd)$ to $\mathcal{W}(\cF L^{t_1},L^{t_2})(\rd)$, with
\begin{equation}\nonumber
\| T f \|_{\mathcal{W}(\cF L^{t_1},L^{t_2})}
\lesssim\|\sigma\|_{M_{v_{s_1,s_2}\otimes1}^{p,q}}\| f \|_{\mathcal{W}(\cF L^{r_1},L^{r_2})}.
\end{equation}
\end{corollary}

\subsection{Action from $\mathcal M^{r_1,r_2}(\R^d)$ to $\mathcal M^{r_2,r_1}(\R^d)$.}\label{3.2}

In this subsection we prove results for tame phase functions that satisfy \eqref{d2fix}.
This setup is particularly  useful to derive fixed-time estimates for a  family of  time-dependent FIOs $\{T_t\}_{t\in\R}$. These arise as  solutions to Cauchy problems for partial differential equations. For instance, consider the propagators $T_t= e^{it H}$,  where $H$ is the Weyl quantization of a quadratic form on the phase space $\rdd$.  We refer e.g. to \cite{cordero2} and \cite{fiomodulation1}.

As a byproduct, we obtain continuity results for FIOs acting between Wiener amalgam spaces (see Corollary \ref{3.15}).

\begin{theorem}\label{elefabioqppqnew}
Consider a tame phase function $\Phi$ that satisfies \eqref{d2fix},
and $1\leq p,q,
r_1,r_2\leq\infty$ such that
\begin{equation}\label{indici}
r_2\leq r_1,\quad q \leq \min(r_2,r_1'),\quad\frac1p+\frac1q\geq 1.
\end{equation}
If the  symbol $\sigma\in
M^{p,q}(\rdd)$, then the corresponding FIO $\ T$ extends to a
bounded operator $\mathcal{M}^{r_1,r_2}(\rd)\to
\mathcal{M}^{r_2,r_1}(\rd)$, with
\begin{equation}\label{normaoperatore9}
\|Tf\|_{\mathcal{M}^{r_2,r_1}} \lesssim
\|\sigma\|_{M^{p,q}}\|f\|_{\mathcal{M}^{r_1,r_2}}.
\end{equation}
\end{theorem}
\begin{proof}
First we observe that if  $\sigma \in M^2(\R^{2d})=L^2(\R^{2d})$ then $T=T_\sigma$ is bounded on $L^2(\rd)$ with
\begin{equation}\label{caseL2}
\| T_\sigma f \|_{L^2(\R^d)} \leq \| \sigma \|_{L^2(\R^{2d})} \| f \|_{L^2(\R^d)},\quad \forall\,f\in\lrd.
\end{equation}
Indeed, using the Cauchy--Schwarz inequality and the Plancherel theorem, for every $f,g\in\lrd$,
\begin{align*}
|\la T_\sigma f,g\ra|&=\left|\la e^{2\pi i\Phi} \sigma,{\bar {\hat{f}}}\otimes g\ra\right|
\leq \| \sigma\|_{L^2(\rdd)}\|f\|_{L^2(\rd)}\|g\|_{L^2(\rd)}
\end{align*}
and \eqref{caseL2} follows.
Next, multilinear complex interpolation between Theorem \ref{elefabioqppq} $(i)$ and  \eqref{caseL2}
yields the estimate \eqref{normaoperatore9}, for $r_2\leq r_1$, $q \leq \min(r_2,r_1')$, $\,p\geq 2$ and $1/p+1/q=1$.
Finally, the inclusion relations for modulation spaces \eqref{modspaceincl1} extend the result to $1\leq p\leq\infty$ and  $1/p+1/q\geq 1$ (note that $q\leq 2$).
\end{proof}
\begin{corollary}\label{3.15}
Under the assumptions of Theorem \ref{elefabioqppqnew},
with \eqref{d2fix} replaced by \eqref{detcond01}, the operator
$T$ extends to a bounded operator from $\mathcal{W}(\cF L^{r_1},L^{r_2})(\rd)$ to $\mathcal{W}(\cF L^{r_2},L^{r_1})(\rd)$, with
\begin{equation}\nonumber
\| T f \|_{\mathcal{W}(\cF L^{r_2},L^{r_1})}
\lesssim\| \sigma \|_{M^{p,q}} \| f \|_{\mathcal{W}(\cF L^{r_1},L^{r_2})}.
\end{equation}
\end{corollary}
The sharpness of the preceeding results can be derived as a special case of the following.

\begin{proposition}\label{necessary0}
Let $1 \leq p,q,r_1,r_2,t_1,t_2 \leq \infty$. Consider the phase function $\Phi\phas=|x|^2/2+x \cdot \eta$ which is tame and satisfies \eqref{d2fix}. Suppose the following estimates for the corresponding FIO $T$:
\begin{equation}\label{normaoperatore8}
\|Tf\|_{\mathcal{M}^{t_1,t_2}} \lesssim \|\sigma\|_{\mathcal M^{p,q}} \|f\|_{\mathcal{M}^{r_1,r_2}},
\quad \forall \sigma \in \mathcal S (\R^{2d}), \quad \forall f \in \mathcal S (\R^d).
\end{equation}
Then
\begin{equation}\label{necessary1}
 \frac1{r_1}-\frac1{t_2} \geq 1-\frac1p-\frac1q, \quad \frac1{r_2}-\frac1{t_2} \geq 1-\frac1p-\frac1q
 \end{equation}
and
\begin{equation}\label{necessary2}
 q \leq \min(t_1,t_2,r_1',r_2').
 \end{equation}
\end{proposition}
\begin{proof}
For $\lambda>0$, consider the family of FIOs $T_\lambda$, having
phase function $\Phi$ and symbols $\sigma_\lambda=\f_{\lambda/\sqrt {2}} \otimes \f_{1/\lambda}$, with $\f(x)=e^{-\pi |x|^2}$ and $\f_{\lambda}(x) = \f(\lambda x)$.
By assumption we have
\begin{equation}\label{assumption3}
\|T_\lambda \f_\lambda\|_{\mathcal{M}^{t_1,t_2}}\lesssim \|\sigma_\lambda\|_{\mathcal M^{p,q}}\|\f_\lambda\|_{\mathcal{M}^{r_1,r_2}}.
\end{equation}
A straightforward computation shows that $T_\lambda\f_\lambda(x)=2^{-d/2}e^{-\pi(\lambda^2-i)|x|^2}$, so that, using\eqref{lem2} with $a=\lambda^2$ and $b=-1$, we obtain
\begin{align*}
\|T_\lambda\f_\lambda\|_{\mathcal M^{t_1,t_2}(\rd)}&\asymp \frac{((\lambda^2+1)^2+1)^{\frac d2\left(\frac 1 {t_1}-\frac12\right)}}{\lambda^{\frac d {t_2}}(\lambda^2(\lambda^2+1)+1)^{\frac d2\left(\frac 1 {t_1}-\frac 1 {t_2}\right)}}\\
&\asymp\left\{\begin{array}{ll}
\lambda^{-\frac {d}{t_2}}, & \lambda\to 0\\
\lambda^{-d\left(1-\frac 1 {t_2}\right)}, & \lambda\to +\infty.
\end{array}
\right.
\end{align*}
From \eqref{gaussdilmod} we obtain
\begin{align}
\|\sigma_\lambda\|_{\mathcal M^{p,q}(\rdd)}&\asymp \|\f_{\lambda/\sqrt {2}}\|_{\mathcal M^{p,q}(\rd)}\|\f_{1/\lambda}\|_{\mathcal M^{p,q}(\rd)} \nonumber \\
& \asymp\left\{\begin{array}{ll}
\lambda^{-d/p}\lambda^{d/q'}=\lambda^{d(1-1/p-1/q)}, & \lambda\to 0 \\
\lambda^{-d/q'}\lambda^{d/p}=\lambda^{-d(1-1/p-1/q)}, & \lambda\to +\infty,
\end{array}
\right. \label{symbolbehaviour2}
\end{align}
whereas $\| \f_\lambda \|_{\mathcal M^{r_1,r_2}}$ depends on $\lambda$ according to \eqref{dilgauss}.
Combining  this with \eqref{assumption3} we obtain
 for $\lambda \to 0$  the inequality
\begin{equation}\nonumber
\frac 1{r_1}-\frac 1 {t_2}\geq 1-\frac 1p-\frac 1q,
\end{equation}
whereas letting $\lambda\to+\infty$ gives
\begin{equation}\nonumber
\frac 1{r_2}-\frac 1 {t_2}\geq 1-\frac 1p-\frac 1q.
\end{equation}
This proves \eqref{necessary1}.

In order to prove \eqref{necessary2}, define $h_\lambda(x) = h(x) e^{- \pi i
\lambda |x|^2}$ for $\lambda \geq 1$ where $h \in
C_c^\infty(\R^d) \setminus \{0\}$, $h \geq 0$, $h$ even, and the parameter-dependent symbol
$\sigma_\lambda=h \otimes h_\lambda$.
Since $\sigma_\lambda $
has support in a compact set independent of $\lambda\geq 1$,
Lemma \ref{lloc} (i) and \cite[Lemma~4.2]{corderonicola4}
give
\begin{equation}\label{symbolbehaviour0}
\| \sigma_\lambda \|_{\mathcal M^{p,q}} \asymp \lambda^{d \left( \frac1{q}
- \frac1{2}\right)}, \quad \lambda \geq 1.
\end{equation}
If we set $f_\lambda=\mathcal F^{-1} (\overline h_\lambda)$, then the operator
with phase function $\Phi$ and symbol $\sigma_\lambda$ acting on $f_\lambda$
is
$$
T f_\lambda (x) = e^{\pi i |x|^2} h(x) \, \mathcal F^{-1} h^2 (x).
$$
Hence $T f_\lambda$ does not depend on $\lambda$, and we may choose $h$ such that
$1 \lesssim \| T f_\lambda \|_{\mathcal M^{t_1,t_2}}$ for $\lambda \geq 1$.
By Lemma \ref{lloc} (ii) and \cite[Lemma~4.2]{corderonicola4} we
have
\begin{equation}\label{functionbehaviour0}
\| f_\lambda \|_{\mathcal M^{r_1,r_2}} \asymp \lambda^{d \left( \frac1{r_1}
- \frac1{2}\right)}.
\end{equation}
Combining \eqref{symbolbehaviour0}, \eqref{functionbehaviour0} with the assumption \eqref{normaoperatore8} and the observation above we obtain
$$
1 \lesssim \lambda^{d \left( \frac1{q}
+\frac1{r_1}- 1\right)}, \quad \lambda \geq 1,
$$
which gives $q \leq r_1'$.

Next we define the symbol $\sigma_\lambda=\chi_n e^{- \pi i |\cdot|^2} \otimes h_\lambda$,
where $\chi_n(x)=\chi(x/n)$, $n>0$ is an integer, $\chi \in \mathcal F C_c^\infty(\R^d)$, $\chi$ real-valued
and $\chi(0)=1$.
If $f=\mathcal F^{-1} h$ we obtain
$$
T f (x) = \chi_n(x) \, \mathcal F^{-1} (e^{-\pi i \lambda |\cdot|^2} h^2) (x).
$$
Again \cite[Lemma~4.2]{corderonicola4} gives
\begin{equation}\label{functionbehaviour2}
\| \mathcal F^{-1} (e^{-\pi i \lambda |\cdot|^2} h^2) \|_{L^{t_1}} \asymp
\lambda^{d \left( \frac1{t_1}
-\frac1{2}\right)}, \quad \lambda \geq 1.
\end{equation}
By means of dominated convergence we know that
$$
\| (1-\chi_n) \, \mathcal F^{-1} (e^{-\pi i \lambda |\cdot|^2} h^2) \|_{L^{t_1}} \leq
\frac1{2} \, \| \mathcal F^{-1} (e^{-\pi i \lambda |\cdot|^2} h^2) \|_{L^{t_1}}
$$
for $n \geq N$ where $N$ is sufficiently large. Let $n \geq N$ be fixed.
We have now
\begin{equation}\nonumber
\begin{aligned}
\lambda^{d \left( \frac1{t_1} -\frac1{2}\right)}
& \asymp  \| \mathcal F^{-1} (e^{-\pi i \lambda |\cdot|^2} h^2) \|_{L^{t_1}} \\
& \leq 2 \left(  \| \mathcal F^{-1} (e^{-\pi i \lambda |\cdot|^2} h^2) \|_{L^{t_1}}
- \| (1-\chi_n) \, \mathcal F^{-1} (e^{-\pi i \lambda |\cdot|^2} h^2) \|_{L^{t_1}} \right) \\
& \leq 2 \, \| \chi_n \, \mathcal F^{-1} (e^{-\pi i \lambda |\cdot|^2} h^2) \|_{L^{t_1}}
= 2 \, \| T f \|_{L^{t_1}}, \quad \lambda \geq 1.
\end{aligned}
\end{equation}
Because $\mathcal F \chi_n$ is supported in a fixed compact set for all $n$, $\mathcal F(T f)$ is supported in a fixed compact set for all $n$ and for all $\lambda \geq 1$.
Thus Lemma \ref{lloc} (ii) gives $\| T f \|_{L^{t_1}} \asymp \| T f \|_{\mathcal M^{t_1,t_2}}$.
Since $\|f\|_{\mathcal{M}^{r_1,r_2}} \lesssim 1$, and \eqref{symbolbehaviour0} holds because $n$ is fixed,
 when combined with the assumption \eqref{normaoperatore8}, this gives
$$
\lambda^{d \left( \frac1{t_1} -\frac1{2} \right)}
\lesssim \lambda^{d \left( \frac1{q} -\frac1{2} \right)}, \quad \lambda \geq 1.
$$
This implies $q \leq t_1$.

Next we note that \eqref{normaoperatore8} and \eqref{duality1} gives
\begin{equation}\label{amalgam0}
\| \tilde T^* f\|_{W(\mathcal F L^{r_1'}, L^{r_2'})} \lesssim \|\sigma\|_{\mathcal M^{p,q}}\|f\|_{W(\mathcal F L^{t_1'}, L^{t_2'})}
\end{equation}
where $\tilde T^*$ is specified by \eqref{Ttildeexpr}.
Let  $\sigma_\lambda=h_\lambda \otimes h$, $f_\lambda=\mathcal F^{-1} (h_\lambda)$.
Then
\begin{equation}\nonumber
\tilde T^* f_\lambda (x) = h(x) \, \mathcal F^{-1} (e^{-\pi i |\cdot|^2} h^2) (x)
\end{equation}
which implies $1 \lesssim \| \tilde T^* f_\lambda \|_{W(\mathcal F L^{r_1'}, L^{r_2'})}$ for all $\lambda \geq 1$.

As before we obtain
\begin{equation}\nonumber
\| f_\lambda \|_{\cW(\mathcal F L^{t_1'}, L^{t_2'})} =  \| h_\lambda \|_{\mathcal M^{t_1',t_2'} }
\asymp \lambda^{d \left( \frac1{t_2'} - \frac1{2}\right)}, \quad \lambda \geq 1.
\end{equation}
Combination with \eqref{amalgam0} and \eqref{symbolbehaviour0} now gives $q \leq t_2$.

Finally, in order to prove \eqref{necessary2}, it remains to verify $q \leq r_2'$.
Let $\sigma_\lambda=h_{-\lambda} \otimes \chi_n$, $f=\mathcal F^{-1} h$.
Then
\begin{equation}\nonumber
\tilde T^* f (x) = \chi_n (x) \, \mathcal F^{-1} (e^{-\pi i (1+\lambda) |\cdot|^2}  h^2) (x),
\end{equation}
The same argument as above and Lemma \ref{lloc} (i) give, for $n$ sufficiently large (and fixed) and $\lambda \geq 1$,
\begin{equation}\nonumber
\begin{aligned}
\lambda^{d \left( \frac1{r_2'} -\frac1{2}\right)} &
\asymp \| \mathcal F^{-1} ( e^{-\pi i (1+\lambda) |\cdot|^2} h^2) \|_{L^{r_2'}} \\
& \lesssim \| \chi_n \, \mathcal F^{-1} ( e^{-\pi i (1+\lambda) |\cdot|^2} h^2) \|_{L^{r_2'}} \\
& = \| \widehat{\chi_n} * e^{-\pi i (1+\lambda) |\cdot|^2} h^2 \|_{\mathcal F L^{r_2'}} \\
& \asymp \| \widehat{\chi_n} * e^{-\pi i (1+\lambda) |\cdot|^2} h^2 \|_{\mathcal M^{r_1',r_2'}} \\
& = \| \tilde T^* f\|_{W(\mathcal F L^{r_1'}, L^{r_2'})}.
\end{aligned}
\end{equation}
The estimate \eqref{symbolbehaviour0} holds since $n$ is fixed and by Lemma \ref{lloc} (i) and \cite[Lemma~4.2]{corderonicola4}
\begin{equation}\nonumber
\| h_{-\lambda} \|_{\mathcal M^{p,q}}
\asymp \| \cF h_{-\lambda} \|_{L^q}
= \| \cF (\overline{h_{\lambda}}) \|_{L^q}
= \| \cF h_{\lambda} \|_{L^q}
\asymp \lambda^{d \left( \frac1{q}
- \frac1{2}\right)}, \quad \lambda \geq 1.
\end{equation}
These considerations combined with \eqref{amalgam0} finally prove $q \leq r_2'$.
\end{proof}
We apply the previous result to discuss the sharpness of Theorem \ref{elefabioqppqnew}.
Indeed, if we choose $t_1=r_2$ and $t_2=r_1$, then \eqref{necessary1} becomes
$$\frac1p+\frac1q\geq 1, \quad \frac1{r_2}-\frac1{r_1} \geq 1- \frac1p-\frac1q $$
and \eqref{necessary2} is $q\leq \min(r_1,r_2,r_1',r_2')$. If we assume $r_2 \leq r_1$,  this can be rephrased as $q\leq \min(r_2,r_1')$, and thus the conditions \eqref{indici} are necessary
under the assumption $r_2 \leq r_1$.

\section{Consequences for pseudodifferential operators}\label{sharpness}

If we choose the  phase function $\Phi(x,\o) = x \cdot
\o$, the FIO reduces to a pseudodifferential operator in
the Kohn--Nirenberg form. Boundedness results for
pseudodifferential operators acting between modulation spaces are
contained in many recent papers, see e.g.
\cite{boulkhemair,CG02,corderonicola4,GH99,GH04,Toft,Toftweight}.
In particular, the action of a pseudodifferential operator
between different modulation spaces was studied by Toft, and his result
can be rephrased in our context as follows
\cite[Theorem~4.3]{Toft}.

\begin{theorem}\label{psidomodtoft} Assume that $1 \leq p,q,r_1,t_1,r_2,t_2 \leq
\infty$ satisfy
\begin{equation}\label{settoft}
1/r_1 - 1/t_1 = 1/r_2 - 1/t_2 = 1 - 1/p - 1/q, \quad q \leq t_1, t_2
\leq p.
\end{equation}
 Then the pseudodifferential operator $T$, from $\cS(\rd)$ to $\cS'(\rd)$,
 having symbol  $\sigma \in M^{p,q}(\R^{2d})$, extends uniquely
to a bounded operator from $\mathcal{M}^{r_1,r_2}(\R^d)$ to
$\mathcal{M}^{t_1,t_2}(\R^d)$, with the estimate
\begin{equation}\label{stimaToft}
\|Tf\|_{\mathcal{M}^{t_1,t_2}} \lesssim
\|\sigma\|_{M^{p,q}}\|f\|_{\mathcal{M}^{r_1,r_2}}.
\end{equation}
\end{theorem}

We can now prove our main result concerning psedodifferential operators, stated in the introduction.
\begin{proof}[Proof of
Theorem \ref{Charpseudo}.]

(i) \emph{Sufficient conditions}. We observe that the assumptions of Theorem
\ref{boundedphasegradient} are satisfied when $\Phi\phas=x\cdot\eta$ and the result  follows immediately.\\
(ii) \emph{Necessary conditions}. We now assume \eqref{stimaA} and want to show that \eqref{indicitutti} and \eqref{indiceq} hold.

For $\lambda>0$, consider the families of pseudodifferential operators of Kohn--Nirenberg form $T_\lambda$, having phase function $\Phi$ and symbol $\sigma_\lambda=\f_{\lambda/\sqrt {2}} \otimes \f_{1/\lambda}$, with $\f(x)=e^{-\pi |x|^2}$ and $\f_{\lambda}(x) = \f(\lambda x)$.
Observe that the behavior of these symbols is expressed by \eqref{symbolbehaviour2} in the proof of Proposition \ref{necessary0}.
By assumption we have
\begin{equation}\label{assumption4}
\|T_\lambda \f_\lambda\|_{\mathcal{M}^{t_1,t_2}}\lesssim \|\sigma_\lambda\|_{\mathcal M^{p,q}}\|\f_\lambda\|_{\mathcal{M}^{r_1,r_2}}.
\end{equation}
Since $T_\lambda\f_\lambda(x)=2^{-d/2}e^{-\pi\lambda^2|x|^2}$ we obtain
by \eqref{gaussdilmod}
\begin{align}\label{resultbehaviour1}
\|T_\lambda\f_\lambda\|_{\mathcal M^{t_1,t_2}(\rd)}
& \asymp\left\{\begin{array}{ll}
\lambda^{-\frac {d}{t_1}}, & \lambda\to 0, \\
\lambda^{-\frac {d}{t_2'}}, & \lambda\to +\infty.
\end{array}
\right.
\end{align}
Combining \eqref{assumption4}, \eqref{resultbehaviour1} with \eqref{symbolbehaviour2}
and \eqref{gaussdilmod}, we obtain for $\lambda \to 0$
\begin{equation}\nonumber
\frac 1{r_1}-\frac 1 {t_1}\geq 1-\frac 1p-\frac 1q,
\end{equation}
whereas letting $\lambda\to+\infty$ gives
\begin{equation}\nonumber
\frac 1{r_2}-\frac 1 {t_2}\geq 1-\frac 1p-\frac 1q.
\end{equation}
This proves \eqref{indicitutti}.

In order to prove \eqref{indiceq} we set
$h_\lambda(x) = h(x) e^{- \pi i \lambda |x|^2}$ for $\lambda \geq 1$, where $h \in
C_c^\infty(\R^d) \setminus \{0\}$, $h \geq 0$, $h$ even, and the parameter-dependent symbol
$\sigma_\lambda=h \otimes h_\lambda$.
If we set $f_\lambda=\mathcal F^{-1} (\overline h_\lambda)$, the operator
with symbol $\sigma_\lambda$ acting on $f_\lambda$
is $T f_\lambda (x) = h(x) \, \mathcal F^{-1} h^2 (x)$. The same argument as in the proof of Proposition \ref{necessary0}
 gives $q \leq r_1'$.

If we switch quantization from the Kohn--Nirenberg $\sigma(x,D)$
to the Weyl quantization $\sigma^w(x,D)$, according to $\sigma(x,D)=a^w(x,D)$, then
we have $\| \sigma \|_{M^{p,q}} \asymp \| a \|_{M^{p,q}}$ (cf.
\cite[Remark~1.5]{Toft}). This fact, in combination with the Weyl
quantization formula
$$
\la \sigma^ w(x,D) f,g\ra_{L^2} = \la f, \overline{\sigma}^w(x,D) g\ra_{L^2}, \quad f,g
\in \mathcal S(\R^d),
$$
allows us to conclude that the assumption \eqref{stimaA} implies
the dual result
\begin{equation}\nonumber
\| T f \|_{\mathcal M^{r_1',r_2'}} \lesssim \| \sigma \|_{M^{p,q}}
\| f \|_{\mathcal M^{t_1',t_2'}}.
\end{equation}
Reasoning as above we now obtain $q \leq t_1$, which proves $q
\leq \min(t_1,r_1')$.

It remains to show $q\leq\min(t_2,r'_2)$.
We use the same arguments as in \cite[Theorem 5.2]{corderonicola4}. Let
us write down the details for the benefit of the reader. We
consider the Weyl quantization $\sigma^w(x,D)$ and conjugate the operator $\sigma^w(x,D)$ with the
Fourier transform. Then $\Fur^{-1}
\sigma^w(x,D) \Fur = (\sigma \circ \chi)^w(x,D)$, where $\chi(x,\o)=(\o,-x)$. Moreover, the
map $\sigma \mapsto \sigma \circ\chi$ is an isomorphism of $M^{p,q}$, so
that \eqref{stimaA} is equivalent to
\begin{equation*}
\| T f \|_{W(\cF L^{t_1},L^{t_2})} \lesssim \| \sigma \|_{M^{p,q}}
\| f \|_{W(\cF L^{r_1},L^{r_2})} \quad \forall \sigma\in\cS(\rdd)
\quad \forall f\in\cS(\rd),
\end{equation*}
where we switch back to the Kohn--Nirenberg form of the operator.
Now we test the last estimate on the same families of symbols
$\sigma_\lambda = h \otimes h_\lambda$ and functions
$f_\lambda = \mathcal F^{-1} (\overline h_\lambda)$ as specified above.
Observe that, since the functions $f_\lambda$ have
Fourier transforms supported in a fixed compact set,
we have by Lemma \ref{lloc} (i) $\| f_\lambda \|_{\cW( \cF L^{r_1},L^{r_2})} = \| \hat f_\lambda \|_{\cM^{r_1,r_2}} \asymp \|f_\lambda\|_{L^{r_2}}$. Hence we get $q \leq r_2'$.
By duality we finally obtain $q\leq t_2$.
\end{proof}

Observe that, for $r_1=r_2$ and $t_1=t_2$, Theorem  \ref{Charpseudo} gives
the sharpness of Theorem \ref{fiointerpolation1}.

Let us conclude by a comparison between our results and those of Theorem \ref{psidomodtoft} (\cite[Theorem~4.3]{Toft}).

The index set corresponding to \eqref{indicitutti} and \eqref{indiceq} is
larger than the index set corresponding to \eqref{settoft}.
For simplicity, let us draw a picture for the
particular case $t_i=r_i$, $i=1,2$. In this case \eqref{settoft}
reduces to $q\leq t_1,t_2 \leq q'=p$,  whereas \eqref{indicitutti} and
\eqref{indiceq} become $q\leq\min(p',t_1,t_2,t_1',t_2')$.
Projections of the set of points
$(1/p,1/q,1/t_1,1/t_2)$ that satisfy $q\leq t_1,t_2 \leq q'=p$ and $q\leq\min(p',t_1,t_2,t_1',t_2')$, respectively, onto the $(1/p,1/q)$-plane and onto the
$(1/q,1/t_i)$-plane, $i=1,2$, are shown in Figures $1$ and $2$,
respectively. We see that the range of exponents specified by \eqref{indicitutti} and
\eqref{indiceq} widens the exponents specified by \eqref{settoft} in the
$(1/p,1/q)$-plane, whereas  the range of exponents in the $(1/q,1/t_i)$-plane remains the same.

\vspace{1.2cm}
 \begin{center}
           \includegraphics{figindicisimboli.1}
            \\
           $ $
\end{center}
 \begin{center}{\quad\quad\quad \quad\quad \,\,(a)\hfill \quad\quad  (b)}\quad\quad\quad\quad\quad\quad\quad\quad
           \end{center}
         \smallskip
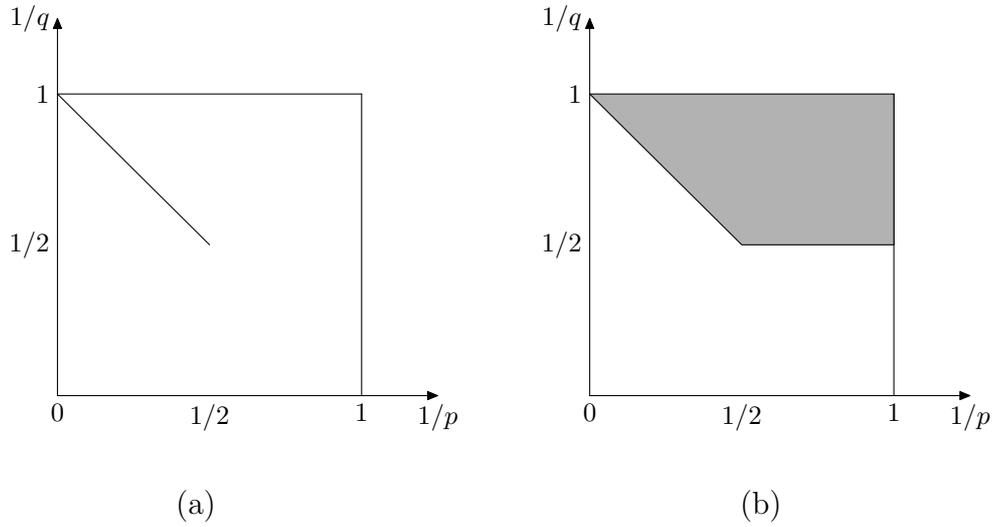
           \begin{center}{ Figure 1. The range of exponents $(1/p,1/q)$ when $t_i=r_i$, $i=1,2$:\\ (a) The range in \eqref{settoft},  (b)  The new range in \eqref{indicitutti} and \eqref{indiceq}. }
           \end{center}
 \begin{center}
           \includegraphics{indiciqt.1}
            \\
           $ $
\end{center}
\begin{center}{ Figure 2. The range of exponents $(1/q,1/r_i)$ when $t_i=r_i$, $i=1,2$: \eqref{settoft}, and  \eqref{indicitutti}, \eqref{indiceq}, coincide. }
           \end{center}

\end{document}